\documentclass[letterpaper, 10 pt, journal, twoside]{IEEEtran}  

\usepackage{graphicx} 
\usepackage{amsmath} 
\usepackage{amssymb}  
\usepackage{booktabs}
\usepackage{tensor}
\usepackage{placeins}
\usepackage{accents}
\usepackage{xcolor}
\usepackage{amsthm}
\usepackage{mathtools}
\usepackage{algorithm}
\usepackage{algorithmic}
{
    \newtheorem{assum}{Assumption}
    \newtheorem{lemma}{Lemma}
    \newtheorem{thrm}{Theorem}
    \theoremstyle{remark}
    \newtheorem{rem}{Remark}
}
\usepackage{balance}
\usepackage{cite}

\newcommand{\bma}[1]{\left[\begin{array}{#1}}
\newcommand{\ema}{\end{array}\right]}

\DeclareMathAlphabet{\mbf}{OT1}{ptm}{b}{n}

\newcommand{\InlineIf}[2]{\STATE \textbf{if} #1 \textbf{then} #2}

\def\fdotb{{\raisebox{-0.6ex}{ \kern0.2ex\raisebox{0.8ex}{\tiny $\hspace*{-1ex}\circ$}}}}
\def\fddotb{{\raisebox{-0.6ex}{ \kern0.2ex\raisebox{0.8ex}{\tiny $\hspace*{-1ex}\circ\circ$}}}}

\newcommand{\utimes}{ {\raisebox{-0.6ex}{ \kern-1.0ex\raisebox{0.6ex}{ \small $\mathsf{v}$}}} } %
\newcommand{\beq}{\begin{equation}}
\newcommand{\eeq}{\end{equation}}
\newcommand{\bdis}{\begin{displaymath}}
\newcommand{\edis}{\end{displaymath}}
\newcommand{\beqarray}{\begin{eqnarray}}
\newcommand{\eeqarray}{\end{eqnarray}}
\newcommand{\beqarraynn}{\begin{eqnarray*}}
\newcommand{\eeqarraynn}{\end{eqnarray*}}

\DeclareMathOperator*{\argmax}{arg\,max}

\makeatletter
\newcommand\fs@betterruled{%
  \def\@fs@cfont{\bfseries}\let\@fs@capt\floatc@ruled
  \def\@fs@pre{\vspace*{5pt}\hrule height.8pt depth0pt \kern2pt}%
  \def\@fs@post{\kern2pt\hrule\relax}%
  \def\@fs@mid{\kern2pt\hrule\kern2pt}%
  \let\@fs@iftopcapt\iftrue}
\floatstyle{betterruled}
\restylefloat{algorithm}
\makeatother

\title{\LARGE \bf Generalized Semi-Infinite Programming for Robust Optimal Control with Decision-Dependent Uncertainty}

\author{Jad Wehbeh$^{1}$, Eric C. Kerrigan$^{2}$ and Edoardo Scaccia$^{3}$
\thanks{$^{1}$Jad Wehbeh is with the Department of Electrical and Electronic Engineering, Imperial College London, SW7 2AZ, UK
        {\tt\small j.wehbeh22@imperial.ac.uk}}%
\thanks{$^{2}$Eric C. Kerrigan is with the Department of Electrical and Electronic Engineering and the Department of Aeronautics, Imperial College London,
        SW7 2AZ, UK
        {\tt\small e.kerrigan@imperial.ac.uk}}%
\thanks{$^{3}$Edoardo Scaccia is with the Department of Electrical and Electronic Engineering, Imperial College London, SW7 2AZ, UK
        {\tt\small e.scaccia24@imperial.ac.uk}}
}

\begin{document}

\FloatBarrier

\maketitle
\thispagestyle{empty}
\pagestyle{empty}


\begin{abstract}
   Generalized semi-infinite programs (GSIPs) arise in robust optimal control whenever the admissible uncertainty depends on the state or controls. Existing GSIP methods either impose restrictive structural assumptions or require global optimization that scales poorly to control problems. We present a general framework that reformulates any GSIP with mild regularity as an existence-constrained semi-infinite program, smoothing its disjunctive feasibility condition into differentiable existence constraints over a fixed index superset. The resulting program is solved by established adaptive discretization (cutting-plane) methods using only off-the-shelf nonlinear-programming solvers, and converges under standard assumptions. Treating the state trajectory as part of the uncertainty extends the framework to robust nonlinear optimal control with state-dependent uncertainty. We demonstrate it on a nonconvex benchmark GSIP and a satellite de-tumbling problem with dynamically varying inertia.
\end{abstract}



\section{Introduction}
\subsection{Background and Motivation}

\IEEEPARstart{S}{emi-infinite programs} (SIPs) are optimization problems with a finite-dimensional decision space but an infinite number of constraints~\cite{lopez2007semi}. Each constraint must hold for every value of an index variable ranging over a fixed set, typically an operating condition or uncertainty, as in robust optimization and robust optimal control~\cite{zagorowska2024automatic}. Classically this index set is fixed and independent of the decision variables; here it depends on the operational decisions (directly through the control inputs or indirectly through the system state), so the index represents a decision-dependent uncertainty. Problems with such state- or control-dependent uncertainties include examples from aviation~\cite{chwa2014fuzzy}, robotics~\cite{navvabi2019new}, and electronics~\cite{zhao2013robust}.

SIPs whose index set depends on the decision variables are generalized SIPs (GSIPs)~\cite{stein2002generalized} and are not always tractable by classical SIP methods. Although direct GSIP methods exist~\cite{vazquez2008generalized}, we instead convert GSIPs into existence-constrained SIPs to reuse existing SIP methods such as those of~\cite{wehbeh2024semi}.
 
\subsection{State of the Art}

Explicit methods for state-dependent uncertainty are uncommon and mostly learning-based: robust or stochastic MPC with data-driven models~\cite{soloperto2018learning, bonzanini2021learning}; adaptive sliding-mode control~\cite{li2022adaptive} is a non-learning exception. Of these, only~\cite{soloperto2018learning} guarantees robust constraint satisfaction, and, unlike SIP-based methods, none optimizes directly against the worst case over the state-dependent set.

Direct semi-infinite programming for state-dependent uncertainties is rare. A notable exception~\cite{hu2025global} is limited to polynomial problems with polyhedral index sets. SIP-based robust optimal control is more common but cannot handle the state-dependent uncertainties considered here~\cite{zagorowska2024automatic}. We therefore examine general techniques for GSIPs, which either solve the GSIP directly or first convert it into a classical SIP~\cite{vazquez2008generalized}. The former~\cite{mitsos2015global, kirst2019global} obtain global solutions by discretization and branch-and-bound, respectively; such global methods are necessary for nonconvex constraints to guarantee feasibility~\cite{djelassi2019discretization}, but scale poorly in the index-set dimension. 

We focus on transforming generalized SIPs into standard SIPs, enabling adaptive discretization techniques applied successfully in related settings~\cite{wehbeh2024semi}. Originally described in~\cite{blankenship1976infinitely} and known as a cutting-plane or constraint-generation method, adaptive discretization replaces the infinite constraint set by finitely many points and scales to the high-dimensional problems typical of predictive control~\cite{hettich1993semi}. The related generalized robust cutting-set method of~\cite{isenberg2021generalized}, implemented in the PyROS solver~\cite{sherman2026pyros}, applies the same discretization to nonconvex two-stage problems with recourse, using only off-the-shelf nonlinear-programming solvers; however, it is cast over an uncertainty set fixed a priori and cannot capture the decision-dependence considered here. Unfortunately, converting generalized SIPs into SIPs is not trivial. Existing methods either rely on restrictive assumptions to define the transformations globally, such as degree 2 stability of the inner maximizer in~\cite{polak2005use}, or relax these assumptions to build more limited local transformations~\cite{still1999generalized}. The simpler transformation introduced here converts generalized SIPs to existence-constrained SIPs, tackled either by the adaptive discretization algorithm below or by the global technique of~\cite{djelassi2021global}.

\subsection{Contributions}
This paper expands on~\cite{wehbeh2025state} to provide an efficient solution methodology for GSIPs and robust optimal control problems (ROCPs) with state-dependent uncertainties. Its main contributions are:

\begin{itemize}
    \item A refined transformation of any suitably regular GSIP into an existence-constrained SIP, smoothing its disjunctive feasibility constraints (Theorem~\ref{thrm:gsip_existence_form}).
    \item Its integration into an adaptive discretization algorithm, convergent under standard assumptions with each subproblem solved globally (Theorem~\ref{thrm:alg_conv}).
    \item A framework casting a broad class of nonlinear ROCPs with decision-dependent uncertainty as generalized SIPs, solved by the same machinery (Theorem~\ref{thm:robusttoSIP}).
    \item Validation on a nonconvex GSIP benchmark~\cite{lemonidis2008global} and a satellite de-tumbling problem with inertial uncertainty.
\end{itemize}

\section{Generalized Semi-Infinite Programming}
\label{sec:gsip}
Consider the generalized SIP of the form
\begin{subequations}
\label{eq:gsip_general}
\begin{align}
    \label{eq:gsip_base_obj}
    & \min_{z \in \mathcal{Z}} \; f(z) \\
    \label{eq:gsip_base_constraint}
    \text{s.t.} \quad 
    & g(z, v) \leq 0 \qquad \forall v \in \mathcal{V}(z)
\end{align}
\end{subequations}
where $z \in \mathcal{Z} \subseteq \mathbb{R}^{n_z}$ is the decision variable, $v \in \bar{\mathcal{V}} \subset \mathbb{R}^{n_v}$ a bounded index variable, $f : \mathcal{Z} \rightarrow \mathbb{R}$ the objective, and $g : \mathcal{Z} \times \bar{\mathcal{V}} \rightarrow \mathbb{R}^{n_g}$ the constraints to be satisfied robustly. The decision-dependent index set~$\mathcal{V}(z)$ is defined as
\begin{equation}
    \label{eq:uncertainty_set}
\mathcal{V}(z) \coloneqq 
\left\{\, v \in \bar{\mathcal{V}} \;\middle|\;
\begin{array}{l}
h(z,v) \geq 0 \\
r(z,v) = 0
\end{array}
\right\}\end{equation}
with $h : \mathcal{Z} \times \bar{\mathcal{V}} \rightarrow \mathbb{R}^{n_h}$ and $r : \mathcal{Z} \times \bar{\mathcal{V}} \rightarrow \mathbb{R}^{n_r}$. Throughout, $[n] \coloneqq \{1,\ldots,n\}$ for a positive integer~$n$.

\begin{assum}
    \label{assum:gsip_assums}
    We assume throughout:
    \begin{itemize}
    \item \emph{(a)~Compactness:}~$\mathcal{Z}$ and~$\mathcal{V}(z)$ are non-empty, compact, and uniformly bounded for all $z \in \mathcal{Z}$.
    \item \emph{(b)~Determined part:} there is a partition $v = (p, q)$ with $p \in \mathbb{R}^{n_r}$, $q \in \bar{\mathcal{Q}} \subset \mathbb{R}^{n_v - n_r}$ and~$\bar{\mathcal{Q}}$ compact, such that, with $Q(z) \coloneqq \{q : \exists\, p,\ (p,q) \in \mathcal{V}(z)\}$, the following hold for every $z \in \mathcal{Z}$: (i)~for each $q \in \bar{\mathcal{Q}}$, $r(z, (p, q)) = 0$ has a unique solution~$p$, continuous in~$(z,q)$; (ii)~$Q(z) \subseteq \bar{\mathcal{Q}}$; and (iii)~every candidate $(p,q)$ with $r(z,(p,q))=0$ and $q \in \bar{\mathcal{Q}}$, including infeasible ones where $h \not\geq 0$, lies in $\operatorname{int}(\bar{\mathcal{V}})$, so that~$g$, $h$, and~$r$ are defined on a neighborhood of every point the algorithm evaluates.
    \item \emph{(c)~Smoothness:}~$f$, $g$, $h$, $r \in \mathcal{C}^\rho$, $\rho \in \{0,\ldots,+\infty\}$ ($\rho$-times differentiable for $\rho \geq 1$, continuous for $\rho=0$).
    \end{itemize}
\end{assum}

These parts play distinct roles: (a)~compactness, with continuity, makes each worst-case search attain its maximum; (b)~the determined part reduces the index to~$q$, giving the standard-SIP reformulation; and (c)~smoothness makes the finite subproblems differentiable, hence solvable by off-the-shelf nonlinear-programming solvers.

\begin{lemma}
    \label{lem:gsip_interm}
    Any generalized SIP satisfying Assumption~\ref{assum:gsip_assums}(b) is equivalent to a semi-infinite program over the fixed index set~$\bar{\mathcal{Q}}$, of the form
\begin{subequations}
    \label{eq:gsip_interm}
\begin{align}
    & \qquad \min_{z \in \mathcal{Z}} \; f(z) \\
    \text{s.t.} \quad & \forall q \in \bar{\mathcal{Q}}, \, v = (p,q) \text{ with } r(z,(p,q)) = 0, \notag \\ 
    \label{eq:gsip_interm_constraint}
    & \left[g(z, v) \leq 0\right] \lor
      \Big[ \min_{j\in [n_h]} h_j(z,v) < 0\Big]
\end{align}
\end{subequations}
\end{lemma}

\begin{proof}
    By~\eqref{eq:uncertainty_set} and the containment in Assumption~\ref{assum:gsip_assums}(b), $v = (p,q) \in \mathcal{V}(z)$ if and only if $q \in \bar{\mathcal{Q}}$, $r(z,(p,q)) = 0$, and $h(z,v) \geq 0$, with~$p$ the unique solution of $r(z,(p,q))=0$.

    Let~$z$ satisfy~\eqref{eq:gsip_base_constraint}, and take any $q \in \bar{\mathcal{Q}}$ and~$p$ with $r(z,(p,q)) = 0$. If $h(z,v) \geq 0$ then $v \in \mathcal{V}(z)$, so $g(z,v) \leq 0$ and the first clause holds; otherwise $\min_{j} h_j(z,v) < 0$ and the second holds. Conversely, suppose~$z$ satisfies~\eqref{eq:gsip_interm_constraint}. For any $v \in \mathcal{V}(z)$ we have $h(z,v) \geq 0$, so the second clause fails and $g(z,v) \leq 0$; hence~$z$ satisfies~\eqref{eq:gsip_base_constraint}.
\end{proof}

We treat the strict inequality in~\eqref{eq:gsip_interm_constraint} as non-strict hereafter.
\begin{rem}[Exactness of the relaxation]
\label{rem:nonstrict_cq}
Let $A(z) \coloneqq \{q \in \bar{\mathcal{Q}} : h(z,(p,q)) \geq 0\}$ be the admissible free indices, with~$p$ the unique root of $r(z,(p,q))=0$. The relaxation is exact whenever the strictly admissible indices $\{q : h(z,(p,q)) > 0\}$ are dense in~$A(z)$: the non-strict form enforces $g(z,v) \leq 0$ there, and continuity from Assumption~\ref{assum:gsip_assums}(b)--(c) extends it to~$A(z)$. This can fail for decision-dependent~$h$: a decision may collapse~$A(z)$ onto its boundary, admitting a point with $\min_{j} h_j(z,v) = 0$ but $g_i(z,v) > 0$ for some~$i$; the satellite of Section~\ref{sec:sat_example} is regularized so that the qualification holds.
\end{rem}

We smooth the disjunction of Lemma~\ref{lem:gsip_interm} into a differentiable existence constraint via the conversion of~\cite{kirjner1998conversion} and the logic-smoothing approaches of~\cite{wehbeh2026exactcontinuousreformulationslogic, wehbeh2025smooth}.

\begin{thrm}
\label{thrm:gsip_existence_form}
    Any generalized SIP satisfying Assumption~\ref{assum:gsip_assums}(b) can be written, in the non-strict sense introduced above and exactly so under the qualification of Remark~\ref{rem:nonstrict_cq}, as the existence-constrained SIP of the form
\begin{subequations}
\label{eq:gsip_smooth}
\begin{align}
    & \qquad \qquad \qquad \min_{z \in \mathcal{Z}} \; f(z) \\
    & \; \text{s.t.} \; \forall q \in \bar{\mathcal{Q}}, \,v = (p,q) \text{ with } r(z,(p,q)) = 0, \notag \\
    \label{eq:gsip_smooth_constraint}
    & \quad \forall\, i \in [n_g], \; \exists\, \lambda_i \in \Lambda_{n_h+1} \;:\; \notag \\
    & \qquad \lambda_{i,1}\, g_i(z, v) + \sum_{j = 1}^{n_h} \lambda_{i,j+1}\, h_j(z, v) \leq 0
\end{align}
\end{subequations}
where $\Lambda_{n_h+1} \coloneqq \big\{ \lambda \in \mathbb{R}^{n_h+1} \mid \sum_{j=1}^{n_h+1} \lambda_j = 1,\, \lambda_j \ge 0 \big\}$ is the probability simplex. With~$p$ retained as a variable pinned by $r(z,(p,q))=0$, the objective and constraints are~$\rho$-times differentiable under Assumption~\ref{assum:gsip_assums}(c).
\end{thrm}
\begin{proof}
    By the non-strict treatment above, the disjunction~\eqref{eq:gsip_interm_constraint} requires, for each $i \in [n_g]$, that $\left[g_i(z,v) \leq 0\right] \lor \left[\min_{j} h_j(z,v) \leq 0\right]$, i.e.\ $\min\{g_i(z,v),\, h_1(z,v),\, \ldots,\, h_{n_h}(z,v)\} \leq 0$. Since a linear function attains its minimum over~$\Lambda_{n_h+1}$ at a vertex, and the vertices are the standard basis vectors, this minimum equals $\min_{\lambda_i \in \Lambda_{n_h+1}} \big( \lambda_{i,1} g_i(z,v) + \sum_{j=1}^{n_h} \lambda_{i,j+1} h_j(z,v) \big)$~\cite{kirjner1998conversion}, so the disjunction is equivalent to~\eqref{eq:gsip_smooth_constraint}. Differentiability follows since~$f$, $g$, $h$, $r$ are~$\mathcal{C}^\rho$ by Assumption~\ref{assum:gsip_assums}(c) and the smoothed constraint is linear in~$\lambda_i$.
\end{proof}

Similar continuous encodings arise, in the homotopy limit, from the logic-smoothing approach of~\cite{malyuta2023fast}. The disjunction~\eqref{eq:gsip_interm_constraint} could instead be handled by a mixed-integer master within the cutting-plane method of~\cite{blankenship1976infinitely}; we smooth it into~\eqref{eq:gsip_smooth} to avoid that combinatorial cost.

\section{Adaptive Discretization for Generalized SIPs}
\label{sec:adaptive_disc}

Problem~\eqref{eq:gsip_smooth} is a semi-infinite program: the index~$q$ ranges over the infinite set~$\bar{\mathcal{Q}}$ and the existence quantifier over the continuous simplex~$\Lambda_{n_h+1}$. We solve it by adaptive discretization~\cite{blankenship1976infinitely}, detailed in Algorithm~\ref{alg:existence_sip}. Problem~\eqref{eq:gsip_smooth} reduces to a standard SIP (Lemma~\ref{lem:bf_reduction}), whose master is exact on any finite working set (Lemma~\ref{lem:master}); feasibility reduces to the sign of a worst-case value (Lemma~\ref{lem:violation}), computed exactly by the separation program~\eqref{eq:q_max} (Lemma~\ref{lem:oracle_exact}); the algorithm converges (Theorem~\ref{thrm:alg_conv}).

\begin{lemma}
    \label{lem:bf_reduction}
    Under Assumption~\ref{assum:gsip_assums}, problem~\eqref{eq:gsip_smooth} is equivalent to the standard semi-infinite program
    \begin{equation*}
        \min_{z \in \mathcal{Z}} f(z) \quad \text{s.t.} \quad G_i(z,q) \leq 0 \quad \forall\, (i,q) \in [n_g] \times \bar{\mathcal{Q}},
    \end{equation*}
    where $G_i(z,q) \coloneqq \min\{g_i, h_1, \ldots, h_{n_h}\}$ at the $v=(p,q)$ solving $r(z,(p,q))=0$; each~$G_i$ is continuous on the compact set $\mathcal{Z} \times \bar{\mathcal{Q}}$.
\end{lemma}
\begin{proof}
    By Assumption~\ref{assum:gsip_assums}(b), $r(z,(p,q)) = 0$ has a unique solution~$p$ for each $q \in \bar{\mathcal{Q}}$, so the index reduces to $q \in \bar{\mathcal{Q}}$. The existence constraint of~\eqref{eq:gsip_smooth} for index~$i$ holds exactly when $\min_{\lambda_i \in \Lambda_{n_h+1}}\big(\lambda_{i,1} g_i + \sum_{j=1}^{n_h} \lambda_{i,j+1} h_j\big) \leq 0$; since a linear function attains its minimum over~$\Lambda_{n_h+1}$ at a vertex, this equals $\min\{g_i, h_1, \ldots, h_{n_h}\} \leq 0$, that is, $G_i(z,q) \leq 0$ at the $v=(p,q)$ solving $r(z,(p,q))=0$. Each~$G_i$ is a minimum of finitely many functions, continuous by~(c) and composed with the continuous solution of~(b), so it is continuous on the compact $\mathcal{Z} \times \bar{\mathcal{Q}}$.
\end{proof}

We retain the existence form~\eqref{eq:gsip_smooth} in the master: the non-smooth~$G_i$ would render the master subproblem nondifferentiable, whereas the simplex multipliers~$\lambda_i$ restore differentiability.

\begin{algorithm}[t]
\begin{algorithmic}[1]
 \caption{\strut Existence-Constrained SIP Solver}
 \label{alg:existence_sip}
 \renewcommand{\algorithmicrequire}{\textbf{Input:}}
 \renewcommand{\algorithmicensure}{\textbf{Output:}}
 \REQUIRE Nonempty finite $\mathbb{Q}_0 \subseteq \bar{\mathcal{Q}}$.
 \ENSURE  $z^*$, $\mathbb{Q}$
 \\ {\color{black}\textit{Initialization} : $\mathbb{Q} \leftarrow \mathbb{Q}_0$}, $\;$ $m \leftarrow \mathrm{card}(\mathbb{Q}_0)$
 \WHILE{\textbf{true}}
    \STATE Solve~\eqref{eq:gsip_lr_nlp} for a minimizer~$z_m^*$
    \STATE $m \leftarrow m + 1$
    \FOR{$i = 1$ to~$n_g$}
        \STATE Solve~\eqref{eq:q_max} to obtain $(q_i^*, p_i^*, s_i^*)$
    \ENDFOR
    \STATE Select $i^* \in \argmax_{i \in [n_g]} s_i^*$ and set $(q_m, s_m) \leftarrow (q_{i^*}^*, s_{i^*}^*)$
    \InlineIf{$s_m \leq 0$}{\textbf{break}}
    \STATE $\mathbb{Q} \leftarrow \mathbb{Q} \cup \{q_m\}$
 \ENDWHILE
 \STATE $z^* \leftarrow z_{m-1}^*$
 \end{algorithmic}
 \end{algorithm}

Fix a finite working set~$\mathbb{Q}$ and restrict the index~$q$ in~\eqref{eq:gsip_smooth} to~$\mathbb{Q}$. With a copy~$p_m$ of the implicit variable per discretization point and a variable $\lambda_{i,m} \in \Lambda_{n_h+1}$ per pair~$(i,m)$, the existence quantifier of~\eqref{eq:gsip_smooth} becomes explicit and the master is the standard nonlinear program
\begin{subequations}
\label{eq:gsip_lr_nlp}
\begin{align}
    & \min_{\substack{z \in \mathcal{Z},\; \{p_m\}_{m \in [\mathrm{card}(\mathbb{Q})]},\\ \{\lambda_{i,m}\}_{i \in [n_g],\, m \in [\mathrm{card}(\mathbb{Q})]}}} \; f(z) \\
    & \text{s.t.} \; \forall m \in [\mathrm{card}(\mathbb{Q})], \; i \in [n_g] : \notag \\
    & \quad v_m = (p_m,q_m), \; \lambda_{i,m} \in \Lambda_{n_h+1}, \; r(z,v_m) = 0, \notag \\
    & \quad \lambda_{i,m,1}\, g_i(z, v_m) + \sum_{j = 1}^{n_h} \lambda_{i,m,j+1}\, h_j(z, v_m) \leq 0,
\end{align}
\end{subequations}
where, under Assumption~\ref{assum:gsip_assums}(c), the objective and constraints of~\eqref{eq:gsip_lr_nlp} are~$\rho$-times differentiable.

\begin{lemma}
\label{lem:master}
For any finite~$\mathbb{Q} \subseteq \bar{\mathcal{Q}}$, a decision~$z$ is feasible for~\eqref{eq:gsip_lr_nlp} if and only if it satisfies the existence constraint of~\eqref{eq:gsip_smooth} at every~$q \in \mathbb{Q}$. Hence~\eqref{eq:gsip_lr_nlp} is a relaxation of~\eqref{eq:gsip_smooth}, exact at the retained points, and its optimal value lower-bounds that of~\eqref{eq:gsip_smooth}.
\end{lemma}
\begin{proof}
Each~$\lambda_{i,m}$ ranges over the full simplex~$\Lambda_{n_h+1}$, so the constraint of~\eqref{eq:gsip_lr_nlp} at~$q_m$ is satisfiable exactly when some $\lambda_i \in \Lambda_{n_h+1}$ makes $\lambda_{i,1}\, g_i(z,v_m) + \sum_{j=1}^{n_h} \lambda_{i,j+1}\, h_j(z,v_m) \leq 0$, which is the existence constraint of~\eqref{eq:gsip_smooth} at~$q_m$, with~$p_m$ pinned by $r(z,v_m)=0$. Since~$\mathbb{Q} \subseteq \bar{\mathcal{Q}}$, restricting the quantifier $\forall q \in \bar{\mathcal{Q}}$ to~$\mathbb{Q}$ enlarges the feasible set, so the optimal value cannot exceed that of~\eqref{eq:gsip_smooth}.
\end{proof}

Certifying feasibility then requires the worst-case violation of each constraint over~$\bar{\mathcal{Q}}$.

\begin{lemma}
\label{lem:violation}
Fix~$z$ and $i \in [n_g]$, and let
\begin{equation*}
    V_i(z) \coloneqq \max_{q \in \bar{\mathcal{Q}}} \; \min_{\lambda_i \in \Lambda_{n_h+1}} \Big( \lambda_{i,1}\, g_i(z, v) + \sum_{j = 1}^{n_h} \lambda_{i,j+1}\, h_j(z, v) \Big),
\end{equation*}
with $v = (p,q)$ and $r(z,v) = 0$. Then~$z$ satisfies the existence constraint of~\eqref{eq:gsip_smooth} for index~$i$ at every $q \in \bar{\mathcal{Q}}$ if and only if $V_i(z) \leq 0$.
\end{lemma}
\begin{proof}
At a fixed~$q$, the existence constraint of~\eqref{eq:gsip_smooth} holds for index~$i$ exactly when some $\lambda_i \in \Lambda_{n_h+1}$ makes the bracketed sum nonpositive, that is, when its minimum over~$\Lambda_{n_h+1}$ is nonpositive. This holds at every $q \in \bar{\mathcal{Q}}$ if and only if its maximum over~$q$, namely~$V_i(z)$, is nonpositive.
\end{proof}

The inner maximization is the separation oracle: for each $i \in [n_g]$, the oracle solves the finite program
\begin{subequations}
\label{eq:q_max}
\begin{align}
&(q_i^*(z),p_i^*(z),s_i^*(z)) \in
\argmax_{q \in \bar{\mathcal{Q}},\, p,\, s \in \mathbb{R}} \; s \\
& \text{s.t.} \; v = (p,q), \; r(z,v) = 0, \notag \\
& \hphantom{\text{s.t.} \;} s \le g_i(z,v), \notag \\
& \hphantom{\text{s.t.} \;} s \le h_j(z,v) \quad \forall\, j \in [n_h].
\end{align}
\end{subequations}
Note that no multipliers are needed here: the master must enforce the disjunction $G_i \leq 0$, for which the simplex multipliers of~\eqref{eq:gsip_lr_nlp} restore differentiability, whereas the oracle only maximizes~$G_i$, which the epigraph~\eqref{eq:q_max} captures with smooth constraints.

\begin{lemma}
\label{lem:oracle_exact}
Under Assumption~\ref{assum:gsip_assums}, for every $z \in \mathcal{Z}$ and $i \in [n_g]$, problem~\eqref{eq:q_max} is a finite, $\rho$-times differentiable nonlinear program that attains its maximum, with optimal value $s_i^*(z) = V_i(z)$; every maximizer~$q_i^*(z)$ attains the maximum defining~$V_i(z)$ in Lemma~\ref{lem:violation}.
\end{lemma}
\begin{proof}
By Assumption~\ref{assum:gsip_assums}(b), each $q \in \bar{\mathcal{Q}}$ pins $v = (p,q)$ through $r(z,(p,q)) = 0$, so the feasible~$s$ at that~$q$ are exactly $s \leq G_i(z,q)$ and~\eqref{eq:q_max} maximizes~$G_i(z,\cdot)$ over~$\bar{\mathcal{Q}}$; the maximum is attained since~$G_i$ is continuous on the compact $\mathcal{Z} \times \bar{\mathcal{Q}}$ by Lemma~\ref{lem:bf_reduction}. From the proof of Lemma~\ref{lem:bf_reduction}, $G_i(z,q)$ equals the inner minimum in~$V_i$, hence $\max_{q \in \bar{\mathcal{Q}}} G_i(z,q) = V_i(z)$ and every maximizer of~\eqref{eq:q_max} attains the maximum in Lemma~\ref{lem:violation}. The objective and constraints $r(z,v) = 0$, $s - g_i(z,v) \leq 0$, and $s - h_j(z,v) \leq 0$ are $\rho$-times differentiable by Assumption~\ref{assum:gsip_assums}(c).
\end{proof}

The worst-case discretization point attaining $\max_{i} s_i^*(z)$ is then added to~$\mathbb{Q}$, and the procedure stops once this largest violation is nonpositive.

\begin{thrm}
\label{thrm:alg_conv}
    Suppose that Assumption~\ref{assum:gsip_assums} holds, that~\eqref{eq:gsip_smooth} is feasible, and that every subproblem~\eqref{eq:gsip_lr_nlp} and~\eqref{eq:q_max} is solved to optimality. Algorithm~\ref{alg:existence_sip} then either terminates with a solution of~\eqref{eq:gsip_smooth} or generates a sequence of master solutions whose accumulation points are all optimal for~\eqref{eq:gsip_smooth}.
\end{thrm}
\begin{proof}
By Lemma~\ref{lem:master}, the master~\eqref{eq:gsip_lr_nlp} is a relaxation of~\eqref{eq:gsip_smooth}, exact at the retained points. Separating a violated index is the nontrivial step: the existence quantifier makes each per-point violation a minimization over the simplex, so each worst-case value~$V_i$ is a max-min over the index and the simplex, which Lemmas~\ref{lem:violation} and~\ref{lem:oracle_exact} evaluate exactly. Each iteration appends a point of~$\bar{\mathcal{Q}}$ attaining $\max_i V_i$, so the algorithm is the cutting-plane method of~\cite{blankenship1976infinitely} on the standard SIP of Lemma~\ref{lem:bf_reduction}. If the algorithm terminates, the largest violation is nonpositive, so the master solution is feasible for~\eqref{eq:gsip_smooth}; since the master lower-bounds~\eqref{eq:gsip_smooth}, this feasible optimum is also optimal for~\eqref{eq:gsip_smooth}. Otherwise, since~$\mathcal{Z}$ is compact, the master minimizers~$z_m^*$ have at least one accumulation point in~$\mathcal{Z}$, and the cutting-plane convergence of~\cite{blankenship1976infinitely} renders every such accumulation point optimal for~\eqref{eq:gsip_smooth}.
\end{proof}

\section{State-Dependent Uncertainties in Robust OCPs}
\label{sec:state_dep_uncert}
Treating the state trajectory as part of the uncertainty turns a discrete-time ROCP with state-dependent uncertainties into a generalized SIP of the form treated in Sections~\ref{sec:gsip} and~\ref{sec:adaptive_disc}. Let $x = (x_0,\ldots,x_N) \in \mathcal{X} \subseteq \mathbb{R}^{n_x}$ be the state trajectory, $u = (u_0,\ldots,u_{N-1}) \in \mathcal{U} \subseteq \mathbb{R}^{n_u}$ the control trajectory of inputs, feedback gains, or other decision variables, and $w = (w_0,\ldots,w_{N-1}) \in \mathbb{R}^{n_w}$ the disturbance trajectory, coupled through the discrete-time dynamics
\begin{equation}
    \label{eq:dynamics}
    \phi(u,x,w) = 0,
\end{equation}
a condensed formulation that stacks the per-step dynamics $\phi_k(x_{k+1}, x_k,u_k,w_k)=0$ with the initial condition. Here $\phi : \mathcal{U} \times \mathcal{X} \times \bar{\mathcal{W}} \rightarrow \mathbb{R}^{n_x}$, with~$\bar{\mathcal{W}}$ a compact superset of the admissible disturbances. If $\phi$ instead discretizes an implicit differential-algebraic equation, its algebraic variables are adjoined to the determined part of the uncertainty and their equations appended to~\eqref{eq:dynamics}; the unique-solvability condition below then requires those equations to determine the algebraic variables uniquely and continuously, which index-one structure provides locally.

The admissible disturbances are themselves state- and decision-dependent, constrained as in~\eqref{eq:uncertainty_set} by $h^{\mathrm c}(u,x,w) \geq 0$ and $r^{\mathrm c}(u,x,w) = 0$, where $h^{\mathrm c} : \mathcal{U} \times \mathcal{X} \times \bar{\mathcal{W}} \rightarrow \mathbb{R}^{n_{h^{\mathrm c}}}$ and $r^{\mathrm c} : \mathcal{U} \times \mathcal{X} \times \bar{\mathcal{W}} \rightarrow \mathbb{R}^{n_{r^{\mathrm c}}}$. Because the state and disturbances are coupled through~\eqref{eq:dynamics} while the disturbance set depends on the state, we take $v = (x,w)$ as the uncertainty and let the dynamics enter as an additional equality in the uncertainty set,
\begin{equation*}
\mathcal{V}(u) \coloneqq 
\left\{\, (x,w) \in \mathcal{X} \times \bar{\mathcal{W}} \;\middle|\;
\begin{array}{l}
\phi(u,x,w) = 0 \\
h^{\mathrm c}(u,x,w) \geq 0 \\
r^{\mathrm c}(u,x,w) = 0
\end{array}
\right\}.
\end{equation*}

Our control problem seeks a trajectory~$u$ minimizing the worst-case cost $J : \mathcal{U} \times \mathcal{X} \times \bar{\mathcal{W}} \rightarrow \mathbb{R}$ over the uncertainty~$\mathcal{V}(u)$, subject to constraints $g^{\mathrm c} : \mathcal{U} \times \mathcal{X} \times \bar{\mathcal{W}} \rightarrow \mathbb{R}^{n_{g^{\mathrm c}}}$ (absorbing terminal or boundary requirements, leaving~$\phi$ to the dynamics):
\begin{subequations}
    \label{eq:rocp}
\begin{align}
    & \qquad \min_{u \in \mathcal{U}} \; \max_{(x,w) \in \mathcal{V}(u)} \; J(u,x,w) \\
    \text{s.t.} \quad & g^{\mathrm c}(u,x,w) \leq 0 \qquad \forall (x,w) \in \mathcal{V}(u)
\end{align}
\end{subequations}
\begin{assum}
\label{assum:control_assums}
Mirroring Assumption~\ref{assum:gsip_assums}: \emph{(a)~compactness:} the sets~$\mathcal{X}$, $\mathcal{U}$, and $\Gamma \subset \mathbb{R}$ are non-empty and compact, and~$\mathcal{V}(u)$ is non-empty and compact for every $u \in \mathcal{U}$; \emph{(b)~determined part:} the uncertainty admits a partition $v=(x,p^{\mathrm c},q^{\mathrm c})$ with $q^{\mathrm c} \in \bar{\mathcal{Q}}^{\mathrm c} \subset \mathbb{R}^{n_w-n_{r^{\mathrm c}}}$ compact and $\{q^{\mathrm c} : \exists\, (x,p^{\mathrm c}),\ (x,p^{\mathrm c},q^{\mathrm c}) \in \mathcal{V}(u)\} \subseteq \bar{\mathcal{Q}}^{\mathrm c}$ for all $u \in \mathcal{U}$, such that, for every $u \in \mathcal{U}$ and $q^{\mathrm c} \in \bar{\mathcal{Q}}^{\mathrm c}$, $\phi(u,x,w)=0$ and $r^{\mathrm c}(u,x,w)=0$ with $w=(p^{\mathrm c},q^{\mathrm c})$ admit a unique solution $(x,p^{\mathrm c})$, continuous in $(u,q^{\mathrm c})$, and~$\mathcal{X}$ and the compact superset~$\bar{\mathcal{W}}$ are large enough that every candidate $v=(x,p^{\mathrm c},q^{\mathrm c})$ with $(x,p^{\mathrm c})$ that solution and $q^{\mathrm c} \in \bar{\mathcal{Q}}^{\mathrm c}$, including infeasible ones (where $h^{\mathrm c} \not\geq 0$), lies in $\operatorname{int}(\mathcal{X} \times \bar{\mathcal{W}})$, so that~$\phi$, $r^{\mathrm c}$, $h^{\mathrm c}$, $g^{\mathrm c}$, and~$J$ are defined on a neighborhood of every point the algorithm evaluates; and \emph{(c)~smoothness:}~$J$, $g^{\mathrm c}$, $h^{\mathrm c}$, $r^{\mathrm c}$, $\phi \in \mathcal{C}^\rho$, $\rho \in \{0,\ldots,+\infty\}$. In addition, for every $u \in \mathcal{U}$ the worst-case cost $\max_{(x,w) \in \mathcal{V}(u)} J(u,x,w)$ lies in~$\Gamma$.
\end{assum}

Although~$x$ is treated as uncertain, by the uniqueness in Assumption~\ref{assum:control_assums}(b) the inner maximization in~\eqref{eq:rocp} is the usual worst case over disturbances, each at the state solving the dynamics.

\begin{thrm}
    \label{thm:robusttoSIP}
    Any problem~\eqref{eq:rocp} satisfying Assumption~\ref{assum:control_assums} can be written as a generalized SIP with~$\rho$-times differentiable objective and constraints, then solved to optimality by Algorithm~\ref{alg:existence_sip} under the assumptions of Theorem~\ref{thrm:alg_conv} and, for equivalence with~\eqref{eq:rocp}, the qualification of Remark~\ref{rem:nonstrict_cq}.
\end{thrm}

\begin{proof}
Introducing $\gamma \in \Gamma$, problem~\eqref{eq:rocp} is equivalent to minimizing~$\gamma$ over $(u,\gamma) \in \mathcal{U} \times \Gamma$ subject to $g^{\mathrm c}(u,x,w) \leq 0$ and $J(u,x,w) - \gamma \leq 0$ for all $(x,w) \in \mathcal{V}(u)$. With $z = (u,\gamma)$, $v = (x,w)$, $\bar{\mathcal{V}} = \mathcal{X} \times \bar{\mathcal{W}}$, and objective $f(z) = \gamma$, this is a generalized SIP of the form~\eqref{eq:gsip_general}--\eqref{eq:uncertainty_set} under the identification $g(z,v) = (g^{\mathrm c}(u,x,w),\, J(u,x,w) - \gamma)$, $h(z,v) = h^{\mathrm c}(u,x,w)$, and $r(z,v) = (\phi(u,x,w),\, r^{\mathrm c}(u,x,w))$, the robust constraint stacking the~$n_{g^{\mathrm c}}$ control constraints with the epigraph bound ($i \in [n_{g^{\mathrm c}}+1]$) and $r=0$ collecting the dynamics~\eqref{eq:dynamics} with~$r^{\mathrm c}=0$. Compactness of \(\mathcal U\) and \(\Gamma\) gives a compact decision set, and the partition \(v=(x,p^{\mathrm c},q^{\mathrm c})\) of Assumption~\ref{assum:control_assums}(b), with \((x,p^{\mathrm c})\) the unique solution of $\phi=0$, $r^{\mathrm c}=0$, supplies the determined part of Assumption~\ref{assum:gsip_assums}(b), with $p=(x,p^{\mathrm c})$ and free index $q=q^{\mathrm c}$. The result then follows from Theorems~\ref{thrm:gsip_existence_form} and~\ref{thrm:alg_conv}.
\end{proof}

Algorithm~\ref{alg:existence_sip} then applies verbatim over a finite $\mathbb{Q}^{\mathrm c} \subseteq \bar{\mathcal{Q}}^{\mathrm c}$, with state and disturbance copies $(x_m,p^{\mathrm c}_m)$ adjoined at each retained~$q^{\mathrm c}_m$ through $\phi=0$, $r^{\mathrm c}=0$.

\section{Numerical Results}

Theorem~\ref{thrm:alg_conv} requires optimal subproblem solutions; as the problems here are nonconvex, we instead use fast local optimizers with multistart.

We implement Algorithm~\ref{alg:existence_sip} in Julia, using JuMP~\cite{Lubin2023} for transcription and Ipopt~\cite{wachter2006implementation} as the local optimizer, on a laptop (Intel Core i7-11370H, 16\,GB RAM). Discretization-point initial guesses are uniformly randomized over their bounds unless stated otherwise.

\subsection{GSIP Example}
\label{sec:gsip_example}
Our first example is Problem 15 of~\cite{lemonidis2008global}, nonlinear and nonconvex in the constraints:
\begin{align*}
    & \min_{z\in[0,2]^2}\quad 
 z_2^2 - 4z_2 \\[4pt]
\text{s.t.}\quad &
z_1\cos v + z_2\sin v - 1 \;\le 0,
\quad \forall v \in \mathcal{V}(z), \\[4pt]
& \mathcal{V}(z) \coloneqq \bigl\{\,v \in [0,\pi] \mid -v^2 - \tfrac{7}{4}z_2 + \tfrac{23}{4} \le 0\,\bigr\}
\end{align*}
with solution $z_1^* = 2$ and $z_2^* = 1.4619$. We choose $\bar{\mathcal{Q}} \coloneqq [0,\pi]$ ($q = v$, no equality constraints), so~$\mathcal{V}(z)$ takes the form~\eqref{eq:uncertainty_set} with $h(z,v) = v^2 + \tfrac{7}{4}z_2 - \tfrac{23}{4}$. Across~$10{,}000$ runs, Algorithm~\ref{alg:existence_sip} converges in 5--6 iterations to tolerance~$10^{-4}$, finding~$z^*$ in $38.96\,$ms on average.
For comparison, the method of~\cite{mitsos2015global} requires $23.46\,$s on the same problem, albeit on older hardware and with additional global optimality guarantees.

\subsection{Satellite De-tumbling}
\label{sec:sat_example}

Our robust control example is the de-tumbling and reorientation of an approximately planar satellite under vibration-induced inertial uncertainty. The vibration model depends only on the angular rates $\omega=(\omega_x,\omega_y,\omega_z)$; $I_{zz}$ is fixed by the perpendicular-axis theorem and off-diagonal inertia is neglected. The uncertainty set is then
\begin{equation*}
    \mathcal{I}(\omega) \coloneqq
    \left\{ (I_{xx},I_{yy},I_{zz}) \in \mathbb{R}^3 \,\middle|\,
    \begin{array}{@{}l@{}}
        \bar{I}_{xx} \le I_{xx} \le \bar{I}_{xx} + \varepsilon + a\omega_x^{2} \\[2pt]
        \bar{I}_{yy} \le I_{yy} \le \bar{I}_{yy} + \varepsilon + a\omega_y^{2} \\[2pt]
        I_{zz} = I_{xx} + I_{yy}
    \end{array}
    \right\}.
\end{equation*}
Here $a\omega^2$ is the vibration-induced, rate-dependent uncertainty and $\varepsilon > 0$ a baseline rate-independent floor, kept below $a\omega^2$ at the operating rates so that state-dependence still governs feasibility. Every admissible interval then has width at least~$\varepsilon$, so the strictly admissible indices are dense in the admissible set (each inertia can be perturbed into the interval interior along the trajectory), Remark~\ref{rem:nonstrict_cq} applies, and the non-strict relaxation is exact. With $\varepsilon = 0$, a control driving $\omega_k \to 0$ collapses the interval to $\{\bar{I}\}$, satisfying the smoothed constraint vacuously and admitting a non-robust control.

At every time step $k \in \{0,\ldots,N\}$, the system state is defined by a unit quaternion~$\mathbf{q}_k$ and an angular velocity~$\omega_k$. These evolve, for $k \in \{0,\ldots,N-1\}$, according to the discrete-time dynamics
\begin{align*}
    \mathbf{q}_{k+1} &= \exp\left((T_s/2)\,\Omega(\omega_k)\right) \mathbf{q}_k \\
    \omega_{k+1} &= \omega_k + T_s\,\mathrm{diag}(I_k)^{-1}\left(\tau_k - \omega_k\times(\mathrm{diag}(I_k)\,\omega_k)\right)
\end{align*}
where~$T_s$ is the discretization time step, $I_k = (I_{{xx}_k},I_{{yy}_k},I_{{zz}_k}) \in \mathcal{I}(\omega_k)$ is the vector of principal moments of inertia, $\tau_k \coloneqq (\tau_{x_k},\tau_{y_k},\tau_{z_k})$ is the control torque at step~$k$, and $\Omega(\omega) \coloneqq \left[\begin{smallmatrix} 0 & -\omega^\top \\ \omega & -[\omega]_\times \end{smallmatrix}\right]$ is the unit quaternion kinematic matrix, with $[\omega]_\times$ the skew-symmetric matrix of~$\omega$~\cite{MarkleyCrassidis2014}. Let $(\mathbf{q}_{k+1},\omega_{k+1}) = F(\mathbf{q}_k,\omega_k,\tau_k,I_k)$ denote these dynamics compactly.

The problem minimizes the sum of squared applied torques while pointing the satellite within an angle~$\Delta\theta_{\max}$ of a reference quaternion~$\mathbf{q}_{\mathrm{ref}}$ by time step~$N$:
\begin{subequations}
\label{eq:sat_example_gsip}
    \begin{align}
        &\qquad \min_{\tau \in \mathcal{T}} \; \sum_{k=0}^{N-1} \tau_k^\top \tau_k \\
    \text{s.t.} \quad & |\operatorname{Re}\left(\mathbf{q}_{\mathrm{ref}}\otimes \mathbf{q}_N^{-1}\right)| \;\ge\; \cos\!\left(\Delta\theta_{\max}/2\right) \notag \\
    & \forall\, I_k \in {\mathcal{I}}(\omega_k), \; k \in \{0,\ldots,N-1\}, \notag \\
    & \quad \text{with } (\mathbf{q}_{k+1},\omega_{k+1}) = F(\mathbf{q}_k,\omega_k,\tau_k,I_k)
    \end{align}
\end{subequations}
where $\operatorname{Re}(\cdot)$ denotes the real part of the quaternion and~$\otimes$ quaternion multiplication. Here $\mathcal{T} \subset \mathbb{R}^{3N}$ is a compact box of admissible torque trajectories $\tau \coloneqq (\tau_0,\ldots,\tau_{N-1})$, large enough that its bounds are inactive at the solution.

We use $N=10$, $T_s=1\,\mathrm{s}$, $a=0.05\,\mathrm{kg\,m^2\,s^2}$, $\varepsilon=0.01\,\mathrm{kg\,m^2}$, $\bar{I}_{xx}=\bar{I}_{yy}=0.1\,\mathrm{kg\,m^2}$, $\mathbf{q}_{\mathrm{ref}}=[1,0,0,0]$, $\Delta\theta_{\max}=15^\circ$, and initial state $\omega_0=[0.5,-0.5,0]\,\mathrm{rad/s}$, $\mathbf{q}_0=[0.760,-0.375,-0.375,0.375]$.

We cast~\eqref{eq:sat_example_gsip} as a generalized SIP via the methodology of Section~\ref{sec:state_dep_uncert} and solve it with Algorithm~\ref{alg:existence_sip}. Here the per-step free part is $q^{\mathrm c} = (I_{xx}, I_{yy})$ and the determined part $(x, I_{zz})$, with per-point copies of the state and~$I_{zz}$ constrained by the dynamics and the perpendicular-axis relation $I_{zz}=I_{xx}+I_{yy}$. The compact~$\bar{\mathcal{Q}}^{\mathrm c}$ collects the inertia bounds at the largest rates attainable over the horizon, finite as~$\mathcal{T}$ and~$\mathcal{X}$ are compact. The algorithm solves for all thirty control decision variables, initialized from a single nominal-inertia point, with multistarted nonconvex subproblems and warm-starting by continuation in~$\varepsilon$; $45$ cutting-plane iterations take $34\,\mathrm{min}$ and attain cost $2.5 \times 10^{-3}$. All $50{,}000$ Monte Carlo runs, including boundary-biased and worst-edge samples, meet the $15^\circ$ bound (mean $5.7^\circ$, maximum $13.3^\circ$); a worst-case search over the admissible inertia nevertheless finds a mixed-inertia sequence with terminal error $19.2^\circ$ that sampling does not reach. Monte Carlo alone thus cannot certify robustness: the bound holds only once the cutting-plane oracle value is nonpositive.

No existing technique addresses this problem at this scale; making the uncertainty set state-independent instead, by bounding the angular-rate components $|\omega_x|,|\omega_y| \le 0.5\,\mathrm{rad/s}$ and enforcing this bound on the controller, renders the problem infeasible in our experiments.

\section{Conclusions}
We presented a transformation that recasts a generalized SIP as an existence-constrained SIP with decision-independent uncertainty, covering nonlinear ROCPs with state-dependent uncertainties as a special case. The reformulation is solved by the same adaptive discretization machinery as a standard SIP, with no solver tailored to the generalized setting.

The satellite study shows the value of modeling this dependence directly: the static, state-independent surrogate is infeasible, whereas the state-dependent formulation meets the bound on every sampled and worst-edge trajectory. Decision-dependence is thus a prerequisite for feasibility, not merely a refinement of the uncertainty model.

The guarantee of Theorem~\ref{thrm:alg_conv} is conditional: it requires Assumption~\ref{assum:gsip_assums}, feasibility of the target problem, and each finite subproblem solved to global optimality, which is not guaranteed for the nonconvex programs here. Robustness is certified only at convergence of the algorithm, once the largest oracle value $\max_i s_i^*(z)$ is nonpositive. Natural next steps are to relax the uniqueness of the determined part, bound the suboptimality of local solves and the number of discretization points, and extend the framework to probabilistic uncertainty and larger problems.
\balance

\section*{ACKNOWLEDGMENTS}

The work of Jad Wehbeh was funded by the Natural Sciences and Engineering Research Council of Canada through a PGS D grant. The work of Edoardo Scaccia was funded by the United Kingdom Engineering and Physical Sciences Research Council and by SLB.



\bibliographystyle{IEEEtran}
\bibliography{IEEEabrv,references.bib}

\begin{thebibliography}{10}
\providecommand{\url}[1]{#1}
\csname url@samestyle\endcsname
\providecommand{\newblock}{\relax}
\providecommand{\bibinfo}[2]{#2}
\providecommand{\BIBentrySTDinterwordspacing}{\spaceskip=0pt\relax}
\providecommand{\BIBentryALTinterwordstretchfactor}{4}
\providecommand{\BIBentryALTinterwordspacing}{\spaceskip=\fontdimen2\font plus
\BIBentryALTinterwordstretchfactor\fontdimen3\font minus \fontdimen4\font\relax}
\providecommand{\BIBforeignlanguage}[2]{{%
\expandafter\ifx\csname l@#1\endcsname\relax
\typeout{** WARNING: IEEEtran.bst: No hyphenation pattern has been}%
\typeout{** loaded for the language `#1'. Using the pattern for}%
\typeout{** the default language instead.}%
\else
\language=\csname l@#1\endcsname
\fi
#2}}
\providecommand{\BIBdecl}{\relax}
\BIBdecl

\bibitem{lopez2007semi}
M.~L{\'o}pez and G.~Still, ``Semi-infinite programming,'' \emph{European journal of operational research}, vol. 180, no.~2, pp. 491--518, 2007.

\bibitem{zagorowska2024automatic}
M.~Zagorowska, P.~Falugi, E.~O'Dwyer, and E.~C. Kerrigan, ``Automatic scenario generation for efficient solution of robust optimal control problems,'' \emph{International Journal of Robust and Nonlinear Control}, vol.~34, no.~2, pp. 1370--1396, 2024.

\bibitem{chwa2014fuzzy}
D.~Chwa, ``Fuzzy adaptive output feedback tracking control of {VTOL} aircraft with uncertain input coupling and input-dependent disturbances,'' \emph{IEEE Transactions on fuzzy systems}, vol.~23, pp. 1505--1518, 2014.

\bibitem{navvabi2019new}
H.~Navvabi and A.~H. Markazi, ``New {AFSMC} method for nonlinear system with state-dependent uncertainty: Application to hexapod robot position control,'' \emph{Journal of Intelligent \& Robotic Systems}, vol.~95, pp. 61--75, 2019.

\bibitem{zhao2013robust}
X.~Zhao, L.~Zhang, P.~Shi, and H.~R. Karimi, ``Robust control of continuous-time systems with state-dependent uncertainties and its application to electronic circuits,'' \emph{IEEE Transactions on Industrial Electronics}, vol.~61, no.~8, pp. 4161--4170, 2013.

\bibitem{stein2002generalized}
O.~Stein and G.~Still, ``On generalized semi-infinite optimization and bilevel optimization,'' \emph{European Journal of Operational Research}, vol. 142, no.~3, pp. 444--462, 2002.

\bibitem{vazquez2008generalized}
F.~G. V{\'a}zquez, J.-J. R{\"u}ckmann, O.~Stein, and G.~Still, ``Generalized semi-infinite programming: a tutorial,'' \emph{Journal of computational and applied mathematics}, vol. 217, no.~2, pp. 394--419, 2008.

\bibitem{wehbeh2024semi}
J.~Wehbeh and E.~C. Kerrigan, ``Semi-infinite programs for robust control and optimization: Efficient solutions and extensions to existence constraints,'' in \emph{8th IFAC Conference on Nonlinear Model Predictive Control NMPC 2024}.\hskip 1em plus 0.5em minus 0.4em\relax IFAC, 2024, pp. 317--322.

\bibitem{soloperto2018learning}
R.~Soloperto, M.~A. M{\"u}ller, S.~Trimpe, and F.~Allg{\"o}wer, ``Learning-based robust model predictive control with state-dependent uncertainty,'' \emph{IFAC-PapersOnLine}, vol.~51, no.~20, pp. 442--447, 2018.

\bibitem{bonzanini2021learning}
A.~D. Bonzanini, D.~B. Graves, and A.~Mesbah, ``Learning-based {SMPC} for reference tracking under state-dependent uncertainty: An application to atmospheric pressure plasma jets for plasma medicine,'' \emph{IEEE Transactions on Control Systems Technology}, vol.~30, no.~2, pp. 611--624, 2021.

\bibitem{li2022adaptive}
P.~Li, D.~Liu, and S.~Baldi, ``Adaptive integral sliding mode control in the presence of state-dependent uncertainty,'' \emph{IEEE/ASME Transactions on Mechatronics}, vol.~27, no.~5, pp. 3885--3895, 2022.

\bibitem{hu2025global}
X.~Hu, J.~Nie, and S.~Zhong, ``A global approach for generalized semi-infinite programs with polyhedral parameter sets,'' \emph{Journal of Optimization Theory and Applications}, vol. 207, no.~3, p.~42, 2025.

\bibitem{mitsos2015global}
A.~Mitsos and A.~Tsoukalas, ``Global optimization of generalized semi-infinite programs via restriction of the right hand side,'' \emph{Journal of Global Optimization}, vol.~61, no.~1, pp. 1--17, 2015.

\bibitem{kirst2019global}
P.~Kirst and O.~Stein, ``Global optimization of generalized semi-infinite programs using disjunctive programming,'' \emph{Journal of Global Optimization}, vol.~73, no.~1, pp. 1--25, 2019.

\bibitem{djelassi2019discretization}
H.~Djelassi, M.~Glass, and A.~Mitsos, ``Discretization-based algorithms for generalized semi-infinite and bilevel programs with coupling equality constraints,'' \emph{Journal of Global Optimization}, vol.~75, no.~2, pp. 341--392, 2019.

\bibitem{blankenship1976infinitely}
J.~W. Blankenship and J.~E. Falk, ``Infinitely constrained optimization problems,'' \emph{Journal of Optimization Theory and Applications}, vol.~19, pp. 261--281, 1976.

\bibitem{hettich1993semi}
R.~Hettich and K.~O. Kortanek, ``Semi-infinite programming: theory, methods, and applications,'' \emph{SIAM Review}, vol.~35, no.~3, pp. 380--429, 1993.

\bibitem{isenberg2021generalized}
N.~M. Isenberg, P.~Akula, J.~C. Eslick, D.~Bhattacharyya, D.~C. Miller, and C.~E. Gounaris, ``A generalized cutting-set approach for nonlinear robust optimization in process systems engineering,'' \emph{AIChE Journal}, vol.~67, no.~5, p. e17175, 2021.

\bibitem{sherman2026pyros}
J.~A.~F. Sherman, N.~M. Isenberg, J.~D. Siirola, and C.~E. Gounaris, ``{PyROS}: The {Pyomo} robust optimization solver,'' Optimization Online, 2026.

\bibitem{polak2005use}
E.~Polak and J.~O. Royset, ``On the use of augmented {L}agrangians in the solution of generalized semi-infinite min-max problems,'' \emph{Computational Optimization and Applications}, vol.~31, no.~2, pp. 173--192, 2005.

\bibitem{still1999generalized}
G.~Still, ``Generalized semi-infinite programming: Theory and methods,'' \emph{European Journal of Operational Research}, vol. 119, no.~2, pp. 301--313, 1999.

\bibitem{djelassi2021global}
H.~Djelassi and A.~Mitsos, ``Global solution of semi-infinite programs with existence constraints,'' \emph{Journal of Optimization Theory and Applications}, vol. 188, pp. 863--881, 2021.

\bibitem{wehbeh2025state}
J.~Wehbeh and E.~C. Kerrigan, ``State-dependent uncertainty modeling in robust optimal control through generalized semi-infinite programming,'' in \emph{2025 33rd Mediterranean Conference on Control and Automation (MED)}.\hskip 1em plus 0.5em minus 0.4em\relax IEEE, 2025, pp. 695--700.

\bibitem{lemonidis2008global}
P.~Lemonidis, ``Global optimization algorithms for semi-infinite and generalized semi-infinite programs,'' Ph.D. dissertation, Massachusetts Institute of Technology, 2008.

\bibitem{kirjner1998conversion}
C.~Kirjner-Neto and E.~Polak, ``On the conversion of optimization problems with max-min constraints to standard optimization problems,'' \emph{SIAM Journal on Optimization}, vol.~8, no.~4, pp. 887--915, 1998.

\bibitem{wehbeh2026exactcontinuousreformulationslogic}
J.~Wehbeh and E.~C. Kerrigan, ``Exact continuous reformulations of logic constraints in nonlinear optimization and optimal control problems,'' arXiv:2601.03906, 2026, submitted to \emph{Automatica}.

\bibitem{wehbeh2025smooth}
------, ``Smooth logic constraints in nonlinear optimization and optimal control problems,'' in \emph{Proceedings of the 64th IEEE Conference on Decision and Control (CDC)}, 2025.

\bibitem{malyuta2023fast}
D.~Malyuta and B.~A{\c{c}}{\i}kme{\c{s}}e, ``Fast homotopy for spacecraft rendezvous trajectory optimization with discrete logic,'' \emph{Journal of Guidance, Control, and Dynamics}, vol.~46, no.~7, pp. 1262--1279, 2023.

\bibitem{Lubin2023}
M.~Lubin, O.~Dowson, J.~{Dias Garcia}, J.~Huchette, B.~Legat, and J.~P. Vielma, ``{JuMP} 1.0: {R}ecent improvements to a modeling language for mathematical optimization,'' \emph{Mathematical Programming Computation}, 2023.

\bibitem{wachter2006implementation}
A.~W{\"a}chter and L.~T. Biegler, ``On the implementation of an interior-point filter line-search algorithm for large-scale nonlinear programming,'' \emph{Mathematical programming}, vol. 106, pp. 25--57, 2006.

\bibitem{MarkleyCrassidis2014}
F.~L. Markley and J.~L. Crassidis, \emph{Fundamentals of Spacecraft Attitude Determination and Control}, ser. Space Technology Library.\hskip 1em plus 0.5em minus 0.4em\relax Springer, 2014, vol.~33.

\end{thebibliography}

\end{document}